\documentclass[12pt]{amsart}
\usepackage{graphicx}
\usepackage{amssymb}
\usepackage{amsmath}
\usepackage{amsthm,amsfonts,bbm}
\usepackage{amscd}
\usepackage{geometry}
\usepackage[all,2cell]{xy}
\usepackage{epsfig,epstopdf}
\usepackage[backref]{hyperref}
\UseAllTwocells \SilentMatrices

\newtheorem{thm}{Theorem}[section]
\newtheorem{cor}[thm]{Corollary}
\newtheorem{lem}[thm]{Lemma}
\theoremstyle{definition}
\newtheorem{defi}[thm]{Definition}
\theoremstyle{remark}
\newtheorem{rmk}[thm]{\bf Remark}
\newtheorem{exm}[thm]{\bf Example}
\numberwithin{equation}{section}
\numberwithin{figure}{section}
\def \b{\bar}

\def\diag{\text{diag}}
\def\D{\mathbb{D}}

\def \E{\mathbb{E}}

\def\la{\lambda}

\def \s{\sigma}
\def \S{\mathbb{S}}
\def \spec{\text{Spec}}
\def \T{\mathbb{T}}

\DeclareMathOperator{\sgn}{sgn}

\begin{document}
\title[Hypergraph Covering and Ramanujan Hypergraph]
{Hypergraph coverings and Ramanujan Hypergraphs}

\author[Y.-M. Song]{Yi-Min Song}
\address{Center for Pure Mathematics, School of Mathematical Sciences, Anhui University, Hefei 230601, P. R. China}
\email{songym@stu.ahu.edu.cn}

\author[Y.-Z. Fan]{Yi-Zheng Fan$^\dagger$}
\address{Center for Pure Mathematics, School of Mathematical Sciences, Anhui University, Hefei 230601, P. R. China}
\email{fanyz@ahu.edu.cn}
\thanks{$^\dagger$The corresponding author.
Supported by National Natural Science Foundation of China (Grant No. 12331012).}

\author[Z. Miao]{Zhengke Miao$^\ddagger$}
\address{Research Institute of Mathematical Science and School of Mathematics and Statistics, Jiangsu Normal University, Xuzhou 221116, P. R. China}
\email{zkmiao@jsnu.edu.cn}
\thanks{$^\ddagger$Supported by National Natural Science Foundation of China (Grant No. 12031018).}



\begin{abstract}
In this paper we investigate Ramanujan hypergraphs by using hypergraph coverings.
 We first show that the spectrum of a $k$-fold covering $\bar{H}$ of a connected hypergraph $H$ contains the spectrum of $H$, and that it is the union of the spectrum of $H$ and the spectrum of an incidence-signed hypergraph with $H$ as underlying hypergraph if $k=2$, which generalizes Bilu-Linial result on graph coverings.
 We give a lower bound for the second largest eigenvalue of a $d$-regular hypergraph by universal cover, which generalizes Alon-Boppana bound on $d$-regular graphs and Feng-Li bound on $(d,r)$-regular hypergraphs.
By using interlacing family of polynomials, we prove that every $(d,r)$-regular hypergraph has a right-sided Ramanujan $2$-covering, and  has a left-sided Ramanujan $2$-covering if the roots of the matching polynomial of its incident graph satisfy some condition.
By Ramanujan $2$-coverings, we prove the existence of some families of infinite many  left-sided or right-sided $(d,r)$-regular Ramanujan hypergraphs under certain conditions on $d$ and $r$.
\end{abstract}

\maketitle

\section{introduction}

A \emph{hypergraph} $H=(V,E)$ consists of a vertex set $V=\{v_1,v_2,\cdots,v_n\}$ denoted by $V(H)$ and an edge set $E=\{e_1,e_2,\cdots,e_m\}$ denoted by $E(H)$, where $e_i\subseteq V$ for $i\in[m]:=\{1,2,\cdots,m\}$.
If $|e_i|=r$ for each $i\in[m]$ and $r\ge2$, then $H$ is called an $r$-\emph{uniform} hypergraph.
For a vertex $v\in V(H)$, denote by $N_H(v)$ or simple $N(v)$ the neighborhood of $v$, i.e., the set of vertices of $H$ adjacent to $v$; and denote by $E_H(v)$ or $E(v)$ the set of edges containing $v$.
The \emph{degree} of $v$, denoted by $d_v$, is the cardinality of the set $E(v)$.
The hypergraph $H$ is called \emph{$d$-regular} if each vertex has degree $d$, and is called \emph{$(d,r)$-regular} if it is $d$-regular and $r$-uniform.
A simple graph is a $2$-uniform hypergraph without multiple edge.

A \emph{homomorphism} from a hypergraph $\bar{H}$ to a hypergraph $H$ is a map $\varpi: V(\bar{H})\to V(H)$ such that $\varpi(e) \in E(H)$ for each $e \in E(\bar{H})$; namely, $\varpi$ maps edges to edges.
So $\varpi$ induces a map denoted by $\tilde{\varpi}$ from $E(\bar{H})$ to $E(H)$, and particularly $\tilde{\varpi}$ maps  $E_{\bar{H}}(\bar{v})$ to $E_H(\varpi(\bar{v}))$ for each vertex $\bar{v} \in V(\bar{H})$.

\begin{defi}\label{cov-def}
A homomorphism $\varpi$ from $\bar{H}$ to $H$ is called a \emph{covering projection} if $\varpi$ is a surjection, and the induced map $\tilde{\varpi}|_{E_{\bar{H}}(\bar{v})}: E_{\bar{H}}(\bar{v}) \to E_H(v)$ is a bijection for each vertex $v \in V(H)$ and each $\bar{v} \in \pi^{-1}(v)$.
\end{defi}

Throughout of this paper, \emph{we always assume that the covering projection $\varpi$ in Definition \ref{cov-def} satisfies the following condition: for any edge $e \in E(\b{H})$, $\varpi|_e: e \to \varpi(e)$ is a bijection so that $e$ and $\varpi(e)$ have the same size}.
Under this assumption, if $\bar{H}$ is $m$-uniform, so is $H$.
If $H$ is connected, then there exists a positive integer $k$ such that each vertex $v$ of $H$ has
$k$ vertices  in its preimage $\pi^{-1}(v)$, and each edge $e$ of $H$ has $k$ edges in $\pi^{-1}(e)$.
In this case, $\pi$ is called a \emph{$k$-fold covering projection} (or $k$-sheeted covering projection) from $\bar{H}$ to $H$, and $\bar{H}$ is called a \emph{$k$-cover} (or \emph{$k$-lift}) of $H$.
In Definition \ref{cov-def}, if both $\bar{H}$ and $H$ are  simple graphs, then the condition $\tilde{\varpi}|_{E_{\b{H}}(\bar{v})}: E_{\b{H}}(\bar{v}) \to E_H(v)$ can be replaced by $\varpi|_{N_{\b{H}}(\bar{v})}: N_{\b{H}}(\bar{v}) \to N_H(v)$.

 Feng and Li \cite{FL} introduced the \emph{adjacency matrix} $A(H)=(a_{uv})$ of a hypergraph $H$, which is defined as follows:
$$ a_{uv}=\left\{ \begin{array}{l} | \{e \in E(H): \{u,v\} \subseteq e\} | \mbox{~if~} u \ne v,\\
0 \mbox{~if~} u = v. \end{array}\right.$$
Note that if $H$ is a simple graph, then $A(H)$ is the usual adjacency matrix of $H$.
As $A(H)$ is real and symmetric, all eigenvalues of $A(H)$ are real and can be arranged as
$$ \la_1(H) \ge \la_2(H) \ge \cdots \ge \la_n(H),$$
where $n$ is the number of vertices of $H$.
By Perron-Frobenius theorem, the largest eigenvalue $\la_1(H)$ is exactly the spectral radius of $A(H)$, denoted by $\rho(H)$.
We note there are other ways to define the eigenvalues for hypergraphs, for example, as the eigenvalues for the adjacency tensor (hypermatrix) \cite{CD,FW,LM3,Lim,Qi} or the Laplacian operators of the simplicial complex consisting of all subsets of edges of the hypergraphs \cite{BGP,DR,HJ,PR}.
In this paper the \emph{spectrum, eigenvalues, and eigenvectors} of a hypergraph always refer to the adjacency matrix of the hypergraph unless specified somewhere.

The aim of this paper is to investigate the Ramanujan hypergraphs by using coverings.
We start the discussion from a simple graph $G$.
Let $$\la(G):=\max\{|\la_i(G)|: \la_i(G) \ne \pm \la_1(G)\},$$ which is called the \emph{second eigenvalue} of $G$.
It is known that  $\la(G)$ is closely related to the expansion property of $G$: the smaller $\la(G)$ is, the better expanding $G$ is; see \cite{Alon, HLW}.
However, $\la(G)$ cannot be arbitrarily small.
In fact, $\la(G) \ge \rho(\T_G)-o_n(1)$ (\cite{Green,Cio}), where $\T_G$ is the universal cover of $G$, and $\rho(\T_G)$ is the spectral radius of $\T_G$.
For example, if $G$ is $d$-regular, then $\T_G$ is an infinite $d$-ary tree denoted by $\T_d$ and $\rho(\T_d)=2 \sqrt{d-1}$ (\cite{GM}), which yields the Alon-Boppana bound (\cite{Alon,Nil}):
\begin{equation}\label{AB}
\la(G) \ge 2 \sqrt{d-1}-o_n(1);
\end{equation}
if $G$ is  $(d,r)$-biregular, then $\mathbb{T}_G$ is an infinite $(d,r)$-biregular tree $\T_{d,r}$ and $\rho(\T_{d,r})=\sqrt{d-1}+\sqrt{r-1}$ (\cite{LS}), which yields the Feng-Li bound (\cite{FL,LS}):
\begin{equation}\label{FengLiB}
\la(G) \ge \sqrt{d-1}+\sqrt{r-1}-o_n(1).
\end{equation}


Following \cite{HLW,LPS}, a graph $G$ is called \emph{Ramanujan} if  $\la(G) \le \rho(\T_G)$.
Bilu and Linial \cite{BL} suggested to construct Ramanujan graphs by Ramanujan  $2$-covering: start with your favorite $d$-regular Ramanujan graph and construct an infinite tower of Ramanujan $2$-coverings.
They conjectured that every (regular) graph has a Ramanujan $2$-covering. This approach turned out to be very useful in the groundbreaking result of Marcus, Spielman and Srivastava \cite{MSS}, who proved that every graph has a one-sided Ramanujan 2-covering. Hall et al. \cite{HPS} generalized the result of \cite{MSS} to coverings of every degree.
This translates to that there are infinitely many $d$-regular bipartite Ramanujan graphs of every degree $d$. The question remains open with respect to fully (i.e. non-bipartite) Ramanujan graphs.

Suppose that $H$ is a connected $(d,r)$-regular hypergraph.
Then $\la_1(H)=d(r-1)$, called the \emph{trivial eigenvalue} of $H$.
Feng and Li \cite{FL} proved that
\begin{equation}\label{FengLiB2} \la_2(H) \ge r-2+ 2\sqrt{(d-1)(r-1)}-o_n(1).
\end{equation}
 Cioab\v{a} et al. \cite{CKMNO}
 obtained an upper bound on the order of $H$ with $\la_2(H)$ bounded by a given value.
Note that the universal cover of $H$, denoted by $\T_H$, is one of the bipartite halves of the infinite $(d,r)$-biregular tree $\T_{d,r}$, denoted by $\D_{d,r}$, which has the spectral radius
 $$\rho(\D_{d,r})=r-2+ 2\sqrt{(d-1)(r-1)}.$$
 So Feng-Li bound (\ref{FengLiB2}) becomes
 $$ \la_2(H) \ge  \rho(\T_H)-o_n(1),$$
 for a $(d,r)$-regular hypergraph $H$.

Li and S\'ole \cite{LS} introduced the notion of Ramanujan hypergraph.
Following the definition,
Dumitriu and Zhu \cite{DZ,DZ2} proved that almost surely all $(d,r)$-regular hypergraphs are nearly Ramanujan, which is an extension of Alon’s second eigenvalue conjecture \cite{Alon} for random regular graphs to random regular hypergraphs.
Motivated by Hall et al. \cite{HPS} of introducing one-sided Ramanujan graphs,
we use left-sided or right-sided Ramanujan hypergraphs to distinguish the conditions for $\la$ in the upper bound or lower bound case.
Note that if $d<r$, then $H$ has an \emph{obvious eigenvalue} $-d$ with multiplicity $|V(H)|-|E(H)|$; see Section \ref{Sec-HM}.

\begin{defi}\cite{LS}\label{RamaGraph}
Let $H$ be a connected $(d,r)$-regular hypergraph.
We call $H$ a \emph{Ramanujan hypergraph} if any non-trivial non-obvious eigenvalue $\la$ of $H$ satisfies
\begin{equation}\label{RamaH} | \la-(r-2)| \le 2\sqrt{(d-1)(r-1)},\end{equation}
 $H$ a  \emph{right-sided Ramanujan hypergraph} if any non-trivial eigenvalue $\la$ of $H$ satisfies
\begin{equation}\label{RamaHL}  \la \le r-2+ 2\sqrt{(d-1)(r-1)},\end{equation}
and $H$ a  \emph{left-sided Ramanujan hypergraph}
if any non-obvious eigenvalue $\la$ of $H$ satisfies
\begin{equation}\label{RamaHR} \la \ge r-2- 2\sqrt{(d-1)(r-1)},\end{equation}
\end{defi}


We summarize the main results of this paper as follows.
Firstly we  generalize Bilu-Linial's result \cite{BL} on graph coverings to hypergraph coverings, which says the spectrum of a $2$-covering $\bar{G}$ of a connected graph $G$ is a multi-set union of the spectrum of $G$ and the spectrum of a signed graph with $G$ as underlying graph.

\begin{thm}\label{inclusion}
Let $H$ be a connected hypergraph and let $\bar{H}$ be a $k$-cover of $H$.
Then, as multi-sets, the spectrum of $\bar{H}$ contains that of $H$; in particular, if $k=2$, then the spectrum of $\bar{H}$ is a multi-set union of the spectrum of $H$ and the spectrum of an incidence-signed hypergraph with $H$ as underlying hypergraph.
\end{thm}

Let $\bar{H}$ be a $2$-cover of an $n$-vertex connected hypergraph $H$.
By Theorem \ref{inclusion}, $n$ eigenvalues of $\bar{H}$ are induced from $H$, called \emph{old eigenvalues}, and the other $n$ eigenvalues are induced form an incidence-signed hypergraph, called \emph{new eigenvalues}.
We introduce the notion of Ramanujan covering for the existence or construction of Ramanujan hypergraphs.

\begin{defi}\label{RamaCov}
Let $\bar{H}$ be a $2$-cover of a connected $(d,r)$-regular hypergraph $H$.
We call $\bar{H}$ is a \emph{Ramanujan covering} of $H$ if all the non-obvious new eigenvalues holds Eq. (\ref{RamaH}),  a \emph{right-sided Ramanujan covering} of $H$ if all new eigenvalue holds (\ref{RamaHL}), and a \emph{left-sided Ramanujan covering} of $H$ if all non-obvious new eigenvalue holds (\ref{RamaHR}).
\end{defi}

We generalize Alon-Boppana bound (\ref{AB}) on $d$-regular graphs and Feng-Li bound (\ref{FengLiB2}) on $(d,r)$-regular hypergraphs to $d$-regular hypergraphs.

\begin{thm}\label{lowB}
Let $H$ be a $d$-regular hypergraph, and let $\T_H$ be the universal cover of $H$.
Then
$$\la_2(H) \ge \rho(\T_H)-o_n(1).$$
\end{thm}

The following theorem can also be derived from Marcus, Spielman and Srivastava \cite{MSS}; see Remark \ref{Conc}. To keep consistence of content, we prove it using the language of hypergraph coverings.

\begin{thm}\label{onesR}
Every connected $(d,r)$-regular hypergraph has a right-sided Ramanujan $2$-covering.
\end{thm}

By Theorem \ref{onesR}, starting from a complete $d$-uniform hypergraph on $d+1$ vertices or an affine plane $AG(2,q)$, we prove the existence of two classes of right-sided Ramanujan hypergraphs.

\begin{thm}\label{2RR}
There exists an infinite family of $(d,d)$-regular (respectively, $(d+1,d)$-regular) right-sided Ramanujan hypergraphs for every $d\ge3$ (respectively, for every prime power $d$).
\end{thm}

By the matching polynomial of the incidence graph of a hypergraph, we give a sufficient condition for the existence of left-sided Ramanujan $2$-covering.
Combing Lemma \ref{mutau}, Theorem \ref{onesL} implies that if $\mu_{B_H}(x)$ has no roots $\mu$ with $0 < |\mu| < \left|\sqrt{d-1}-\sqrt{r-1}\right|$, then $H$ has a left-sided Ramanujan $2$-covering.

\begin{thm}\label{onesL}
Let $H$ be a $(d,r)$-regular hypergraph, and $\mu_{\tau}(B_H)$ be the $\tau_H$-th largest root of the matching polynomial $\mu_{B_H}(x)$ of the incidence graph $B_H$ of $H$, where $\tau_H=\min\{\nu(H),e(H)\}$.
If \begin{equation}\label{conj}
\mu_{\tau_H}(B_H)\ge \left|\sqrt{d-1}-\sqrt{r-1}\right|,
\end{equation}
then  $H$ has a left-sided Ramanujan $2$-covering.
\end{thm}

Lastly, by construction from affine planes, we proved the existence of a class of left-sided Ramanujan hypergraphs.

\begin{thm}\label{LR-const}
For any prime power $q$ greater than $4$, there exist infinitely many $(q+1,q)$-regular left-sided Ramanujan hypergraphs.
\end{thm}

Brito et al. \cite{BDH} proved the spectral gaps for $\la_2(B_H)$ and
$\la_{\tau}(B_H)$ which implies almost surely all $(d,r)$-regular hypergraphs are nearly right-sided or left-sided Ramanujan.
The remaining part of the paper is as follows.
 In Section 2, we give some notions and results for later discussion.
 In Section 3 we prove the spectral inclusion or union property between a covering hypergraph and its underlying hypergraph.
In Section 4, we present a lower bound for the second largest eigenvalue of a hypergraph by universal cover.
In the last section, we use Ramanujan coverings to prove the existence of some families of right-sided or left-sided Ramanujan hypergraphs.

\section{preliminaries}

\subsection{Hypergraphs and matrices}\label{Sec-HM}
Let $H$ be a hypergraph.
Denote by $\nu(H)$ the number of vertices of $H$ and $e(H)$ the number of edges of $H$.
A {\it walk} $W$ of length $l$ in $H$ is a sequence of alternate vertices and edges: $v_{0}e_{1}v_{1}e_{2}\cdots e_{l}v_{l}$,
    where $v_{i} \ne v_{i+1}$ and $\{v_{i},v_{i+1}\}\subseteq e_{i+1}$ for $i=0,1,\ldots,l-1$.
The walk $W$ is called a \emph{path} if no vertices or edges are repeated in the sequence, a \emph{cycle} if $v_0=v_l$ and no vertices or edges are repeated in the sequence except $v_0=v_l$, and is \emph{non-backtracking} if $e_{i} \ne e_{i+1}$ for all $i=1,\ldots,l-1$.
In the case of $H$ being a simple graph, we simply write $W: v_0 v_1 \cdots v_l$ as each edge contains exactly two vertices.
The hypergraph $H$ is said to be  \emph{connected} if any two vertices are connected by a walk, and is called a \emph{hypertree} if it is connected and contains no cycles.

The \emph{dual hypergraph} $H^*$ of $H$ is the hypergraph with vertex set $E(H)$ and edge set $\{\{e \in E(H): v \in e\}: v \in V(H)\}$.
The \emph{incident (bipartite) graph} of $H$, denoted by $B_H$, is a bipartite graph with vertex set $V(H) \cup E(H)$ and edge set $\{\{v,e\}: v\in V(H), e \in E(H), v \in e\}$.
If $H$ is a simple graph, then $B_H$ is exactly the subdivision of $H$ (namely, the graph obtained by inserting an additional vertex into each edge of $H$).

The \emph{incident matrix} of $H$ is denoted and defined by $Z(H)=(z_{ve})$ such that $z_{ve}=1$ if $v \in e$ and $0$ else.
Then we have the adjacency matrix of $B_H$ as follows:
$$ A(B_H)=\left(\begin{array}{cc} O & Z(H) \\ Z(H)^\top & O \end{array} \right).$$
The \emph{(signless) Laplacian} of $H$, denoted by $Q(H)$, is defined by
\[ Q(H)=Z(H)Z(H)^\top=D(H)+A(H),
\]
where $D(H)=\diag \{d_v: v \in V(H)\}$, the degree matrix of $H$.
It is known that $Z_{H^*}=Z(H)^\top$, and $B_{H^*}$ is isomorphic to $B_H$.
We also have
$$
A(B_H)^2=\left(\begin{array}{cc} Z(H)Z(H)^\top & O \\
O & Z(H)^\top Z(H)\end{array} \right)=
\left(\begin{array}{cc} Q(H) & O \\ O & Q(H^*) \end{array} \right).$$

Let $\psi_A(\la)$ denote the characteristic polynomials of square matrix $A$.
Then
\begin{equation}\label{polyR} \la^{\nu(H)} \psi_{A(B_H)}(\la)= \la^{e(H)} \psi_{Q(H)}(\la^2),
~ \la^{e(H)} \psi_{Q(H)}(\la)=\la^{\nu(H)} \psi_{Q(H^*)}(\la).
\end{equation}
So, the spectrum of one of $A(B_H), Q(H)$ and $Q(H^*)$ yields the spectra of the other two.
In particular, if $H$ is $(d,r)$-regular, then
$Q(H)=dI + A(H)$ and $Q(H^*)=rI+A(H^*)$.
In this case, the eigenvalues of $A(B_H)$ are the square roots of the eigenvalues of $A(H)$ plus $d$, and the eigenvalues of $A(H)$ is at least $-d$ as $Q(H)$ is positive semidefinite.
If $\nu(H)>e(H)$, $A(H)$ has the \emph{obvious eigenvalue} $-d$ with multiplicity $\nu(H)-e(H)$.
Also, $A(B_H)$ has \emph{obvious eigenvalue} $0$ with multiplicity $|\nu(H)-e(H)|$ if $\nu(H) \ne e(H)$.

\subsection{Finite covering and universal covering}
Gross and Tucker \cite{GT} applied the permutation voltage graphs to characterize the finite coverings of simple graphs. Let $D$ be digraph possibly with multiple arcs, and let $\mathbb{S}_k$ be the symmetric group on the set $[k]$.
 Let $\phi:E(D)\to\mathbb{S}_k$, which assigns a permutation to each arc of $D$.
 The pair $(D, \phi)$ is called a \emph{permutation voltage digraph}.
 A \emph{derived digraph} $D^\phi$ associated with $(D, \phi)$ is a digraph with vertex set $V(D)\times[k]$ such that $((u, i),(v, j))$ is an arc of $D^\phi$ if and only if $(u, v)\in E(D)$ and $i=\phi(u, v)(j)$.

Let $G$ be a simple graph, and let $\overleftrightarrow{G}$ denote the symmetric digraph obtained from $G$ by replacing each edge $\{u, v\}$ by two arcs with opposite directions, written as $e=(u, v)$ and $e^{-1}:= (v, u)$ respectively.
Let $\phi: E(\overleftrightarrow{G})\to\mathbb{S}_k$ be a permutation assignment of $\overleftrightarrow{G}$ which holds that $\phi(e)^{-1}=\phi(e^{-1})$ for each arc $e$ of $\overleftrightarrow{G}$.
The pair $(G, \phi)$ is called a \emph{permutation voltage graph}.
The derived digraph $\overleftrightarrow{G}^\phi$, simply written as $G^\phi$, has symmetric arcs by definition, and is considered as a simple graph.

\begin{lem}\cite{GT}\label{per}
Let $G$ be a connected graph and let $\bar{G}$ be a $k$-cover of $G$. Then there exists an assignment $\phi$ of permutations in $\mathbb{S}_k$ on $G$ such that $G^\phi$ is isomorphic to $\bar{G}$.
\end{lem}

The coverings of a hypergraph can be obtained by permutation assignments via the incident graph of the hypergraph.
 Let $H$ be a connected hypergraph and let $B_H$ be its incident graph.
 Let $\phi: E(\overleftrightarrow{B_H}) \to \S_k$ such that $\phi(v,e)=\phi^{-1}(e,v)$ for each pair $(v,e)$ with $v \in e$.
 By definition, We will get a $k$-covering $B_H^\phi$ of $B_H$.
Note that in the graph $B_H^\phi$, $(e,j)$ is adjacent to the vertices $(u, \phi(u,e)j)$ for all $u \in e$.
Then we will get a $k$-fold covering of $H$ arisen from $B_H^\phi$, denoted by $H_B^\phi$, which has vertex set $V(H) \times [k]$, and edges still denoted by $(e,j):=\{(u, \phi(e,u)j): u \in e\}$ for all $e \in E(H)$ and $j \in [k]$;
see Fig. \ref{FC} for the construction of a $2$-fold covering of a hypergraph.

\begin{figure}
\centering
\includegraphics[scale=.8]{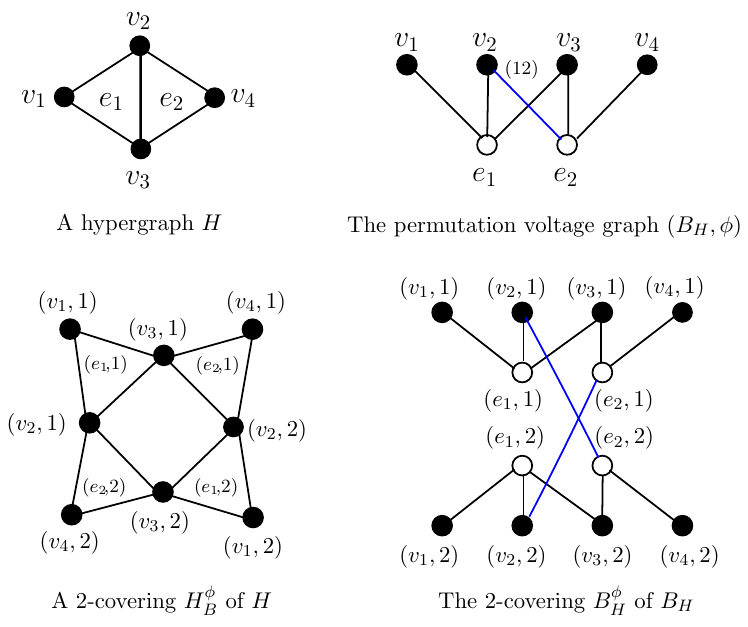}
\caption{\small Construction of a $2$-covering of a hypergraph, where an edge of $H$ or $H^\phi$ is represented by the point set of a triangle, and only the blue edge in $(B_H,\phi)$ is assigned to the permutation $(12)$ and all other edges are assigned to the identity}\label{FC}
\end{figure}

Li and Hou \cite{LH} proved that any $k$-cover $\bar{H}$ of a connected hypergraph $H$ is isomorphic to $H_B^\phi$ for some $\phi$.
Song, Fan et al. \cite{Song} emphasized that there is an isomorphism from $B_H^\phi$ to $B_{\bar{H}}$ which sends $V(H) \times [k]$ to $V(\bar{H})$ and $E(H)\times [k]$ to $E(\bar{H})$, and hence $H_B^\phi$ is isomorphic to $H$.
We give a proof here to simplify the one given in \cite{Song}.

\begin{lem}\cite{LH,Song}\label{hypcov}
Let $\bar{H}$ be a $k$-fold covering of a connected hypergraph $H$. Then there exists a permutation voltage assignment $\phi$ in $\mathbb{S}_k$ on $B_H$ such that $B^\phi_H$ is isomorphic to $B_{\bar{H}}$ by a map which sends $V(H)\times[k]$ to $V(\bar{H})$ and $E(H)\times [k]$ to $E(\bar{H})$, and hence $H^\phi_B$ is isomorphic to $\bar{H}$.
\end{lem}

\begin{proof}
As $\bar{H}$ is a $k$-cover of $H$, $B_{\bar{H}}$ is a $k$-cover of $B_H$ via a covering projection $\varpi$ satisfying $\varpi(V(\bar{H}))=V(H)$ and $\varpi(E(\bar{H}))=E(H)$ (or see the first part of the proof of \cite[Lemma 3.3]{Song}).
By Lemma \ref{per}, there  exists a permutation assignment $\phi: E(\overleftrightarrow{B_H}) \to \mathbb{S}_k$ such that $B_H^\phi$ is isomorphic to $B_{\bar{H}}$ via a map $\eta$.
Note that $B_H^\phi$ is a $k$-cover of $B_H$ via a natural covering projection
$p$ such that $p|_{\{v\}\times [k]}=v$ and $p|_{\{e\}\times [k]}=e$ for all $v \in V(H)$ and $e \in E(H)$.
From the proof of Lemma \ref{per} (see \cite[Theorem 2]{GT}),
$ \varpi \eta=p.$
So, for any $(v,i) \in V(H) \times [k]$, $\eta(v,i) \in V(\bar{H})$;
otherwise, if $\eta(v,i) \in E(\bar{H})$, then $\varpi\eta(v,i) \in E(H)$, which yields a contradiction as $\varpi \eta(v,i)=p(v,i)=v \in V(H)$.
Therefore, $\eta$ sends $V(H)\times[k]$ to $V(\bar{H})$, and $E(H)\times [k]$ to $E(\bar{H})$ by a similar discussion.
\end{proof}

The universal cover of a simple graph $G$
is the unique infinite tree $\mathbb{T}_G$ such that every connected covering of $G$ is a quotient of $\mathbb{T}_G$.
To construct the universal cover, fix a ``root'' vertex $v_0$ of $G$, and
then place one vertex in $\mathbb{T}_G$ for each non-backtracking walk  starting at $v_0$ of every length $l \in \mathbb{N}$.
Two vertices are adjacent in $\mathbb{T}_G$ if their lengths differ by $1$ and one is a prefix of another.
The universal cover
of a graph is unique up to isomorphism and is independent of the choice of the root $v_0$.
In particular, if $G$ is a $d$-regular graph, then $\mathbb{T}_G$ is the infinite $d$-ary tree $\T_d$; if $G$ is a $(d,r)$-biregular graph, then $\mathbb{T}_G$ is the infinite $(d,r)$-biregular tree $\T_{d,r}$,
 where a graph is called $(d,r)$-\emph{biregular} if it is bipartite and the vertices in one part of the bipartition have degree $d$ and the vertices on the other part have degree $r$.

The universal cover $\T_H$ of a hypergraph $H$ is a hypertree, which can be constructed as follows (see \cite{LS}).
First we have the universal cover $\T_{B_H}$ of the incident graph $B_H$ of $H$.
Observe that $\T_{B_H}$ contains two kinds of levels: one consisting of the walks of same length ending at vertices of $H$ (called the \emph{vertex levels}) and the other consisting of the walks of same length ending at the edges of $H$ (called the \emph{edge levels}).
Take the vertices in the vertex levels as the vertices of $\T_H$, and for each vertex $e$ in the edge levels of $\T_{B_H}$,  form an edge of $\T_H$ by taking all vertices of $\T_H$  that is adjacent to $e$; see Fig. \ref{UC} for the universal cover of the hypergraph in Fig. \ref{FC}.
So,  $\T_H$ is the bipartite half of $\T_{B_H}$ corresponding to the vertex levels.
If $H$ is a $(d,r)$-regular hypergraph, then $B_H$ is a $(d,r)$-biregular graph the universal cover of which is $\T_{d,r}$, and hence $\T_H$ is the bipartite half of $\T_{d,r}$ corresponding to the vertex levels, denoted by $\D_{d,r}$.

\begin{figure}
\centering
\includegraphics[scale=.8]{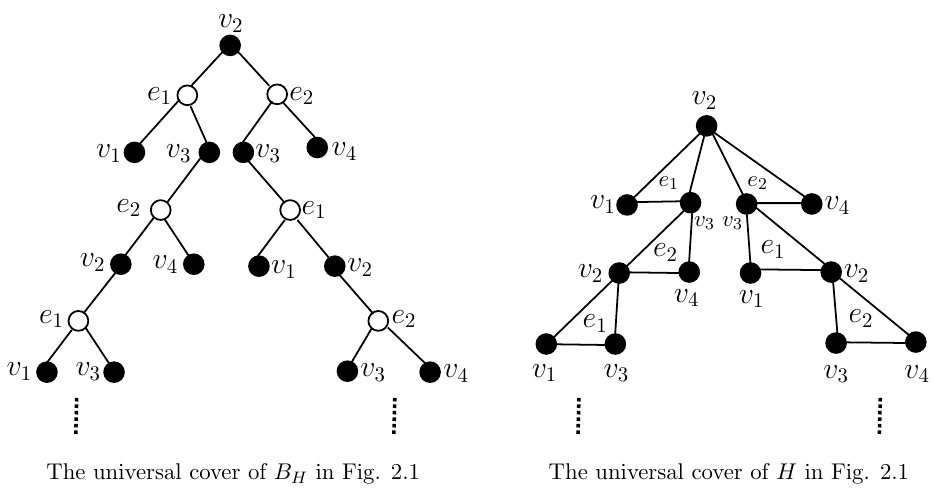}
\caption{\small Construction of the universal cover of the hypergraph $H$ in Fig. \ref{FC} via the universal cover of its incident graph $B_H$, where in the left graph (the infinite tree $\T_{B_H}$), a vertex represents a non-backtracking walk of $B_H$ from the root $v_2$ to the vertex which is the unique path of $\T_{B_H}$ from $v_2$ to the vertex, e.g. the vertex $v_1$ in the $7$th level represents the walk: $v_2e_1v_3e_2v_2e_1v_1$, a unique path of $\T_{B_H}$ from the root to $v_1$ in the $7$th level}\label{UC}
\end{figure}

\subsection{Matching polynomial} A \emph{matching} in a graph is a subset of its edges, no two of which share a common vertex. For a graph $G$, let $m_i$ denote the number of \emph{$i$-matchings} (i.e. a matching consisting of $i$ edges) of $G$ (with $m_0=1$). Heilmann and Lieb \cite{HL} defined the \emph{matching polynomial} of $G$ to be the polynomial
$$\mu_G(x)=\sum_{i\ge0}(-1)^im_i x^{n-2i},$$
where $n$ is the number of vertices of the graph $G$.
The matching polynomial of $G$ is closely related to the characteristic polynomial of the path tree $T_G(u)$ of $G$.
For a graph $G$ and a vertex $u$ of $G$, the \emph{path tree} $T_G(u)$ contains one vertex for each path beginning at $u$, and  two vertices (paths) are adjacent if their lengths differ by $1$ and one is a prefix of another.
Godsil \cite{God81} proved that the matching polynomial $\mu_G(x)$ of $G$ divides the characteristic polynomial of $A(T_G(u))$, which implies that all roots of $\mu_G(x)$ are real and have absolute value at most the spectral radius of $A(T_G(u))$.

Let $\mu(G)$ denote the largest root of $\mu_G(x)$.
By definition, the path tree $T_G(u)$ is a finite induced subgraph of the universal cover $\T_G$.
By Lemma 3.6 of \cite{MSS},
\begin{equation}\label{MatUni}\mu(G) \le \rho(T_G(u)) \le \rho(\T_G).\end{equation}
We now recall an identity of Godsil and Gutman \cite{GodGut} that relates the expected characteristic polynomial over uniformly random signings of the adjacency matrix of a graph to its matching polynomial.
Let $G$ be a graph with $m$ edges ordered as $e_1, \ldots, e_m$.
Associate a signing $s\in\{\pm1\}^m$ to the edges of $G$ such that the sign of $e_i$ is $s_i$ for each $i \in [m]$.
Let $A_s$ denote the signed adjacency matrix corresponding to $s$ and define
\begin{equation}\label{CharSgn}
\psi_s(x)=\mathrm{det}(xI-A_s)
\end{equation}
 to be the characteristic polynomial of $A_s$.

\begin{thm} \cite{GodGut} \label{Match1}Let $G$ be a simple graph and let $\psi_s(x)$ be defined as in (\ref{CharSgn}). Then
$$\E_{s\in\{\pm1\}^m}[\psi_s(x)]=\mu_G(x).$$
\end{thm}

\subsection{Interlacing families} Marcus, Spielman and Srivastava. \cite{MSS} defined interlacing families and examined their properties.

\begin{defi}\cite{MSS}
We say that a real-rooted polynomial $g(x)=\prod^{n-1}_{i=1}(x-\alpha_i)$ \emph{interlaces} a real-rooted polynomial $f(x)=\prod^{n}_{i=1}(x-\beta_i)$ if
$$\beta_1\le\alpha_1\le\beta_2\le\alpha_2\le\cdots\le\alpha_{n-1}\le\beta_n.$$
We say that polynomials $f_1,\ldots, f_k$ have a \emph{common interlacing} if there is a single polynomial $g$ that interlaces each $f_i$ for $i \in [k]$.
\end{defi}

By applying a similar discussion as in \cite[Lemma 4.2]{MSS}, we have the following result.

\begin{lem}\label{MSS4.2}
Let $f_1,\ldots, f_k$ be degree-$n$ real-rooted polynomials with positive leading coefficients, and define
$$f_\emptyset=\sum_{i=1}^kf_i.$$
If $f_1,\ldots, f_k$ have a common interlacing, then for each $\ell \in [n]$ there exists an $i$ for which the $\ell$-th largest root of $f_i$ is at most (at least)  the $\ell$-th largest root of $f_\emptyset$.
\end{lem}

\begin{defi}\cite{MSS}
Let $S_1,\ldots, S_m$ be finite sets, and for each assignment $(s_1,\ldots, s_m)\in S_1 \times \cdots \times S_m$, let $f_{s_1,\ldots,s_m}(x)$ be a real-rooted degree $n$ polynomial with positive leading coefficient.
For a partial assignment $(s_1,\ldots, s_k)\in S_1 \times \cdots \times S_k$ with $k < m$, define
$$ f_{s_1,\ldots,s_k}=\sum_{s_{k+1}\in S_{k+1}, \ldots, s_m \in S_m} f_{s_1,\ldots, s_k,s_{k+1},\ldots,s_m}$$
as well as
$$f_\emptyset=\sum_{s_1\in S_1, \ldots, s_m \in S_m} f_{s_1,\ldots, s_m}.$$
We say $\{f_{s_1,\ldots,s_m}\}_{S_1,\ldots,S_m}$ form an \emph{interlacing family} if for all $k=0,1,\ldots,m-1$ and all $(s_1,\ldots, s_k)\in S_1 \times \cdots \times S_k$, the polynomials $\{f_{s_1,\ldots,s_k,t}\}_{t \in S_{k+1}}$ have a common interlacing.
\end{defi}

By applying a similar discussion as \cite[Theorem 4.4]{MSS}, we have the following result.

\begin{lem}\label{Match2}
Let $S_1, \cdots, S_m$ be finite sets, and let $\{f_{s_1, \ldots, s_m}\}$ be an interlacing family of polynomials of degree $n$. Then for each $\ell \in [n]$ there exists some $s_1, \ldots, s_m\in S_1\times\cdots\times S_m$ so that the $\ell$-th largest root of $f_{s_1, \ldots, s_m}$ is at most (at least) the $\ell$-th largest root of $f_\emptyset$.
\end{lem}

Marcus, Spielman and Srivastava \cite{MSS} proved that the characteristic polynomials $\{\psi_s\}_{s \in \{\pm 1\}^m}$ defined in (\ref{CharSgn}) form an interlacing family.
So, similar to Theorem 5.3 of \cite{MSS}, combining Theorem \ref{Match1} and Eq. (\ref{MatUni}), we have the following result.

\begin{cor}\label{Match3}
Let $G$ be a simple graph on $n$ vertices with adjacency matrix $A$.
Then for each $\ell \in [n]$ there exists a signing $s$ of $A$ so that the $\ell$-th largest eigenvalue of $A_s$ is at most (at least) the $\ell$-th largest root of the matching polynomial of $G$.
In particular, there exists a signing $s$ of $A$ so that the largest eigenvalue of $A_s$ is at most $\rho(\T_G)$.
\end{cor}

\section{Spectra of hypergraph coverings}
In this section, we will investigate the spectral inclusion and union property of hypergraph coverings.
Let $H$ be a hypergraph and $B_H$ be its incidence graph.
For the permutation voltage graph $(B_H, \phi)$, where $\phi: E(\overleftrightarrow{B_H}) \to \S_k$.
Noting that $\sgn \phi(u,e)=\sgn \phi(e,u)$ for each incidence $(u,e)$ with $u \in e$,
define a \emph{signed incident graph} also denoted by $(B_H, \phi)$  such that
each edge $\{u,e\}$ of $B_H$ is given a sign $\sgn \phi(u,e)$,
 and a \emph{signed incident matrix} $Z(H,\phi)$ such that
$Z(H,\phi)_{ue}=\sgn (u,e)$ if $u \in e$, and $0$ else.

The \emph{incidence-signed hypergraph} $(H,\phi)$ is the hypergraph $H$ together with a signing of the incidences such that each incidence $(u,e)$ with $u \in e$ has a sign $\sgn \phi(u,e)$.
The adjacency matrix $ A(H,\phi)$ of $(H,\phi)$ is defined to be
\[
A(H,\phi)_{uv}=\sum_{e \in E(H): \{u,v\} \subseteq e} \sgn \phi(u,e) \sgn \phi(v,e),
\]
and the \emph{Laplacian matrix}  of $(H,\phi)$ is defined to be
$$Q(H,\phi):=Z(H,\phi) Z(H,\phi)^\top = D(H)+A(H,\phi).$$
If $H$ is a simple graph, then $A(H,\phi)$ is the adjacency matrix of a signed graph with $H$ as underlying graph such that each edge $e=\{u,v\}$ has a sign $\phi(u,e) \sgn \phi(v,e)$.
%
%

We get the following relationship between the spectrum of a covering hypergraph $H_B^\phi$ and that of its underlying hypergraph $H$,
where $\spec A$ denotes the spectrum of a square matrix $A$.

\begin{lem}\label{spec}
Let $H_B^\phi$ be a $k$-cover of a hypergraph $H$, where $\phi: E(\overleftrightarrow{B_H}) \to \S_k$.
Then, as multi-sets, the spectrum of $A(H_B^\phi)$ contains that of $A(H)$; in particular, if $k=2$, then, as multi-set union, $\spec A(H_B^\phi)=\spec A(H) \cup \spec A(H,\phi)$.
\end{lem}

\begin{proof}
Let $x=(x_u)$ be an eigenvector of $A(H)$ associated with an eigenvalue $\lambda$. Let $\tilde{x}=(x_{u,i})$ be defined on $V(H) \times [k]$ such that  $\tilde{x}|_{\{u\}\times [k]}:=x_u$ for all $u\in V(H)$.
We assert $\tilde{x}$ is an eigenvector of $A(H_B^\phi)$ associated with the same eigenvalue $\lambda$. By the eigenvector equation, for each $u\in V(H)$,
\begin{equation}\label{eig_H}
\lambda x_u=\sum_{e\in E_H(u):\{u,v\}\subseteq e}x_v.
\end{equation}
Let $A(H_B^\phi):=(a_{(u,i)(v,j)})$.
Observe that $$a_{(u,i)(v,j)}=|\{(e,l) \in E(H) \times [k]: \{(u,i),(v,j)\} \subseteq (e,l)\}|,$$
and by the definition of $H_B^\phi$, $\{u,v\}\subseteq e$ and $l=\phi(e,u)i=\phi(e,v)j$.
So $$a_{(u,i)(v,j)}=|\{e\in E(H): \{u,v\} \subseteq e, \phi(e,u)i=\phi(e,v)j\}|.$$
Now, for any $u\in V(H)$ and $i\in[k]$, by Eq. (\ref{eig_H}) and the definition of $\tilde{x}$, we have
\begin{align*} (A(H_B^\phi)\tilde{x})_{(u,i)}&=\sum_{(v,j) \in V(H) \times [k]} a_{(u,i)(v,j)}\tilde{x}_{v,j}\\
&=\sum_{e\in E_H(u): \{u,v\}\subseteq e,\phi(e,v)j=\phi(e,u)i}\tilde{x}_{v,j}\\
&=\sum_{e\in E_H(u):\{u,v\}\subseteq e}\tilde{x}_{v,\phi(e,v)^{-1}\phi(e,u)i}\\
&=\sum_{e\in E_H(u):\{u,v\}\subseteq e}x_v\\
&=\lambda x_u\\
&=\lambda \tilde{x}_{(u,i)},
\end{align*}
which implies that $\tilde{x}$ is an eigenvector of $A(H_B^\phi)$ associated with the eigenvalue $\lambda$.

If $k=2$, let $y$ be an eigenvector of $A(H,\phi)$ associated with an eigenvalue $\mu$.
Let $\tilde{y}$ be defined on $V(H) \times \{1,2\}$ such that  $\tilde{y}_{u,1}=y_u$ and $\tilde{y}_{u,2}=-y_u$ for all $u\in V(H)$.
We assert that $\tilde{y}$ is an eigenvector of $A(H^\phi)$ associated with the same eigenvalue $\mu$.
By the eigenvector equation, for each $u\in V(H)$,
\begin{equation}\label{eig_Hp} \mu y_u =\sum_{e\in E_H(u):\{u,v\}\subseteq e}\sgn \phi(u,e) \sgn \phi(v,e)y_v.
\end{equation}

For any $u \in V(H)$,
\begin{align*}
(A(H_B^\phi)\tilde{y})_{u,1}&=\sum_{(v,j) \in V(H) \times \{1,2\}} a_{(u,1)(v,j)}\tilde{y}_{v,j}\\
&=\sum_{e\in E_H(u): \{u,v\}\subseteq e,\phi(e,v)j=\phi(e,u)1}\tilde{y}_{v,j}\\
&=\sum_{e\in E_H(u):\{u,v\}\subseteq e}\tilde{y}_{v,\phi(e,v)^{-1}\phi(e,u)1}\\
&=\sum_{e\in E_H(u):\{u,v\}\subseteq e}\sgn \phi(u,e)\sgn \phi(v,e)y_v\\
&=\mu y_u\\
&=\mu\tilde{y}_{u,1};
\end{align*}
and
\begin{align*} (A(H_B^\phi)\tilde{y})_{(u,2)}&=\sum_{(v,j)\in V(H) \times \{1,2\}} a_{(u,2)(v,j)}\tilde{y}_{v,j}\\
&=\sum_{e\in E_H(u):\{u,v\}\subseteq e,\phi(e,v)j=\phi(e,u)2}\tilde{y}_{v,j}\\
&=\sum_{e\in E_H(u):\{u,v\}\subseteq e}\tilde{y}_{v,\phi^{-1}(e,v)\phi(e,u)2}\\
&=\sum_{e\in E_H(u),\{u,v\}\subseteq e}-\sgn \phi(u,e)\phi(v,e)y_v\\
&=-\mu y_u\\
&=\mu\tilde{y}_{u,2},
\end{align*}
which implies that $\tilde{y}$ is an eigenvector of $A(H_B^\phi)$ associated with the eigenvalue $\mu$.

By the above discussion, if $x$ is an eigenvector of $A(H)$, then by an arrangement of the vertices, $(x,\ldots,x)$ (repeating $k$ times) is an eigenvector of $A(H_B^\phi)$.
So, the spectrum of $A(H_B^\phi)$ contains that of $A(H)$, including multiplicity.
If $k=2$, if $y$ is an eigenvector of $A(H,\phi)$, then $(y,-y)$ is an eigenvector of the $2$-covering $H^\phi$.
Noting that $(x,x)$ is orthogonal to $(y,-y)$, so, as multi-set union,
$\spec A(H_B^\phi)=\spec A(H) \cup \spec A(H,\phi)$ in this case.
\end{proof}

\begin{rmk} Lemma \ref{spec} also holds for Laplacian matrix, namely,
under the condition of Theorem \ref{spec},
$\spec Q(H_B^\phi) \supseteq \spec Q(H)$, and if $k=2$, $\spec Q(H_B^\phi) = \spec Q(H) \cup \spec Q(H,\phi)$.
We omit the details here.
\end{rmk}

\begin{proof}[Proof of Theorem \ref{inclusion}]
By Lemma \ref{hypcov}, any $k$-cover $\bar{H}$ of a connected hypergraph $H$ is isomorphic to $H_B^\phi$ for some $\phi$, we immediately get the desired result by Lemma \ref{spec}.
\end{proof}

\section{The second largest eigenvalue}
Let $G$ be a graph or hypergraph (not necessarily finite but local finite with bounded degrees), and let $\mathbb{C}^{V(G)}=\{(x_v)_{v \in V(G)}: x_v \in \mathbb{C}, v \in V(G)\}$.
Define an \emph{inner product} $\langle x, y \rangle=\sum_{v \in V(G)} x_v \bar{y}_v$,
and the \emph{induced norm} $\|x\|_2=\langle x, x \rangle^{1/2}$, where $\bar{y}_v$ denotes the conjugate of $y_v$.
Let $$\ell^2(G)=\{x \in \mathbb{C}^{V(G)}: \|x\|_2 < +\infty\}.$$
Then $\ell^2(G)$ is a complex Hilbert space.
The \emph{adjacency operator} $A(G)$ of $G$ is defined to be such that for each $u \in V(G)$
\begin{equation}\label{Aop}(A(G) x)_u= \sum_{e \in E_G(u): \{u,v\} \in E(G)} x_v.\end{equation}
So $A(G)$ is a bounded linear operator on $\ell^2(G)$.
As $A(G)$ is self-adjoint, the spectral radius $\rho(G)$ of $A(G)$ can be defined as (see \cite{Mohar})
\begin{equation}\label{norm} \rho(G) =\|A(G)\|=\sup_{\|x\|_2=1} |\langle A(G)x, x \rangle|=\sup_{\|x\|_2=1} \langle A(G)x, x \rangle.\end{equation}

Let $H$ be a hypergraph, $B_H$ be its incident graph, and $\T_{B_H}$ the universal cover of $B_H$.
Then we will have a linear operator $A(\T_{B_H})$ on $\ell^2(\T_{B_H})$ defined as in (\ref{Aop}).
Recall that $\T_{B_H}$ has a root vertex say $v_0 \in V(H)$, and contains two kinds of vertex levels: the vertex levels and the edge levels.
Denote by $V_1,V_2$ the sets of vertices in the vertex levels and the edge levels, respectively.
With a little abusing of notations, for a vertex $w$ of $V_1$ (respectively, $V_2$) in $\T_{B_H}$, namely a non-backtracking walk $W$ of $B_H$ starting from $v_0$ and ending at a vertex $v$ of $H$ (respectively, an edge $e$ of $H$), we wills simply use the vertex $v$ (respectively, the edge $e$) to denote the walk $w$; see the labeling of the vertices in the left graph in Fig. \ref{UC}.

We have a decomposition:
$$ \ell^2(\T_{B_H})=\mathbb{C}^{V_1} \oplus \mathbb{C}^{V_2}.$$
Define a linear operator
\begin{equation}\label{Z}Z: \mathbb{C}^{V_1} \to \mathbb{C}^{V_2}
\end{equation}
such that
$$ (Zx)_e=\sum_{v: v \in e} x_v;$$
and
\begin{equation}\label{Z*} Z^*: \mathbb{C}^{V_2} \to \mathbb{C}^{V_1}
\end{equation}
such that
$$ (Z^*x)_v=\sum_{e: v \in e} x_e.$$
It is easily seen that $Z^*$ is the adjoint operator of $Z$, and
$$ A(\T_{B_H})^2=Z^*Z \oplus ZZ^*.$$
As $A(\T_{B_H})$ is self-adjoint and $\rho(Z^*Z)=\rho(ZZ^*)$, we have
\begin{equation}\label{Rerad} \rho(A(\T_{B_H}))^2=\rho(A(\T_{B_H})^2)=\rho(Z^*Z)=\rho(ZZ^*).
\end{equation}

\begin{lem}\label{OpR}
Let $H$ be a hypergraph, let $Z,Z^*$ be defined in (\ref{Z}) and (\ref{Z*}) respectively.
Then
$$ Z^*Z=D(\T_H)+A(\T_H),$$
where $D(\T_H)$ is the degree diagonal operator such that $(D(\T_H)x)_v=d_v x_v$ for each $x \in \ell^2(\T_H)$ and each vertex $v \in V(\T_H)$.
\end{lem}

\begin{proof}
Let $V_1,V_2$ be the sets of vertices in the vertex levels and the edge levels of $\T_{B_H}$, respectively.
Then, for any $x \in \mathbb{C}^{V_1}$ and $v \in V_1$,
\begin{align*} (Z^*Zx)_v & =\sum_{e: v \in e} (Zx)_e\\
& =\sum_{e: v \in e} \sum_{u: u \in e} x_u\\
&=\sum_{e: v \in e} \left(x_v+ \sum_{u: u \in e, u \ne v} x_u\right)\\
&=d_v x_v+ \sum_{e\in E_G(v): \{v,u\} \subseteq e} x_u\\
& =(D(\T_H)x)_v + (A(\T_H)x)_v.
\end{align*}
The result follows.
\end{proof}

Note that $Z^*Z$ is a self-adjoint and positive operator on $\ell^2(\mathbb{C}^{V_1})$.
So by Lemma \ref{OpR},
\begin{equation}\label{Radius} \rho(Z^*Z)=\sup_{\|x\|_2=1} \langle (Z^*Z)x, x \rangle
=\sup_{\|x\|_2=1} \left(\langle D(\T_H) x, x \rangle + \langle A(\T_H) x, x \rangle\right).
\end{equation}
Greenberg \cite{Green} (see also \cite{Cio}) proved that for a graph $G$ on $n$ vertices,
\begin{equation}\label{Green}\la(G) \ge \rho(\T_G)-o_n(1).\end{equation}

\begin{proof}[Proof of Theorem \ref{lowB}]
As $H$ is $d$-regular, we have
$A(B_H)^2=dI + A(H)$ and $D(\T_H)=dI$, where $I$ is an identity matrix or identity operator.
By (\ref{Green}), (\ref{Rerad}) and Lemma \ref{OpR},
\begin{align*}
 \la_2(H)& =\la_2(B_H)^2 -d \\
 &\ge (\rho(\T_{B_H})-o_n(1))^2-d\\
& =\rho(\T_{B_H})^2-d-o_n(1)\\
&=\rho(\T_{B_H}^2)-d-o_n(1)\\
& =\rho(Z^*Z)-d-o_n(1)\\
&=d+\rho(\T_H)-d-o_n(1)\\
&=\rho(\T_H)-o_n(1).
\end{align*}
\end{proof}

We immediately arrive at Feng-Li bound \ref{FengLiB2}.
\begin{cor}\cite{LS}
Let $H$ be a $(d,r)$-regular hypergraph.
Then
$$\la_2(H) \ge  r-2+ 2\sqrt{(d-1)(r-1)}-o_n(1),$$
\end{cor}

\begin{proof}
If $H$ is $(d,r)$-regular, then
$\T_{B_H}$ is a $(d,r)$-biregular infinite tree with spectral radius
$\rho(\T_{B_H})=\sqrt{d-1}+\sqrt{r-1}$.
By (\ref{Rerad}) and (\ref{Radius}), we have
$$\rho(\T_H)=\rho(\T_{B_H})^2-d=r-2+ 2\sqrt{(d-1)(r-1)}.$$
The result follows by Theorem \ref{lowB}.
\end{proof}

\section{Ramanujan hypergraphs}
In this section, we will investigate the existence or the construction of Ramanujan hypergraphs by coverings.
Let $H_B^\phi$ be a $k$-fold covering of a hypergraph $H$, where $\phi: E(\overleftrightarrow{B_H}) \to \S_k$.
By definition,
$$A(B_H,\phi)=\left[ \begin{array}{cc}
O & Z(H,\phi) \\
Z(H,\phi)^\top & O
\end{array} \right],$$
and
$$Q(H,\phi)=Z(H,\phi) Z(H,\phi)^\top = D(H)+A(H,\phi).$$

Similar to Eq. (\ref{polyR}), we have
$$\la^{\nu(H)}\psi_{A(B_H,\phi)}(\la) = \la^{e(H)}\psi_{Q(H,\phi)}(\la^2).$$
Now taking $k=2$ and considering the expectation on both sides over all choices of $\phi$ in $\S_2$ independently and uniformly, by Theorem \ref{Match1}, we have
\begin{equation}\label{exprel} \la^{\nu(H)}\mu_{B_H}(\la)=\la^{\nu(H)} \E (\psi_{A(B_H,\phi)}(\la)) = \la^{e(H)}\E(\psi_{Q(H,\phi)}(\la^2)),
\end{equation}
as each $\phi: E(\overleftrightarrow{B_H}) \to \S_2$ is one to one corresponding to a signing of $B_H$ which yields a signed adjacency matrix $A(B_H,\phi)$.

\subsection{Right-sided Ramanujan hypergraphs}

\begin{lem}\label{Extsgn}
There exits an assignment $\phi: E(\overleftrightarrow{B_H}) \to \S_2$ such that
$$ \la_1(Q(H,\phi)) \le \rho(\T_{B_H})^2.$$
In particular, if $H$ is $(d,r)$-regular, then there exits a $\phi: E(\overleftrightarrow{B_H}) \to \S_2$ such that
$$ \la_1(A(H,\phi)) \le r-2+ 2\sqrt{(d-1)(r-1)}.$$
\end{lem}

\begin{proof}
By Eq. (\ref{exprel}) and Corollary \ref{Match3}, there exits an assignment $\phi$ such that
the largest eigenvalue of $A(B_H,\phi)$ is at most $\rho(\T_{B_H})$.
So,
 $$\la_1(Q(H,\phi))=\la_1(A(B_H,\phi))^2 \le \rho(\T_{B_H})^2.$$

If  $H$ is $(d,r)$-regular, then
$\la_1(Q(H,\phi))=d+\la_1(A(H,\phi)$, and $\rho(\T_{B_H})=\sqrt{d-1}+\sqrt{r-1}$.
Hence,
\begin{align*}
\la_1(A(H,\phi)& =\la_1(Q(H,\phi))-d \\
&  \le \left(\sqrt{d-1}+\sqrt{r-1}\right)^2-d\\
&=r-2+ 2\sqrt{(d-1)(r-1)}.
\end{align*}
The result follows.
\end{proof}

\begin{proof}[Proof of Theorem \ref{onesR}]
Let $H$ be a $(d,r)$-regular hypergraph.
By Lemma \ref{Extsgn}, there is a $\phi: E(\overleftrightarrow{B_H}) \to \S_2$ such that
$ \la_1(A(H,\phi)) \le r-2+ 2\sqrt{(d-1)(r-1)}.$
By definition, $H_B^\phi$ is a $2$-fold covering of $H$.
By Theorem \ref{spec}, $\spec A(H_B^\phi)=\spec A(H) \cup \spec A(H,\phi)$ as multi-set union.
So, $H^\phi_B$ is a right-sided Ramanujan covering of $H$ by definition.
\end{proof}

\begin{exm}\label{exm}(Examples of  Ramanujan hypergraphs)
Let $K^d_{d+1}$ be a complete $d$-uniform hypergraph on $d+1$ vertices, which is obviously $(d,d)$-regular.
The eigenvalues of $A(K^d_{d+1})$ are
$d(d-1)$, and $-(d-1)$ with multiplicity $d$.
It is easily verified that $K^d_{d+1}$ is a  Ramanujan hypergraph.

A \emph{projective plane} $PG(2,q)$ of order $q$ consists of a set $X$ of
$q^2 + q + 1$ elements called points, and a set $B$ of $(q + 1)$-subsets of $X$ called lines, such that any two points lie on a unique line. It can be derived from the definition that any
point lies on $q + 1$ lines, and two lines meet in a unique point.
Associated with the projective plane we have a hypergraph also denoted by $PG(2,q)$ whose vertices are the points of $X$ and edges are the lines of $B$.
Then $PG(2,q)$ is a $(q + 1,q+1)$-regular
 hypergraph with $q^2 + q + 1$ vertices.
It is easily to verify that $PG(2,q)$ has the trivial eigenvalue $q(q+1)$, and the eigenvalues $-1$ with multiplicity $q(q+1)$.
So $PG(2,q)$ is a Ramanujan hypergraph.

An \emph{affine plane} $AG(2, q)$ of order $q$ can be obtained from the projective plane $PG(2,q)$ by  removing a line and all the points on it.
It is easily seen that the corresponding hypergraph $AG(2, q)$ is a $(q+1,q)$-regular, which has $q^2$ vertices and $q^2+q$ edges such that any
two point lie in a unique line.
The eigenvalues of $AG(2, q)$ are $q^2-1$, and $-1$ with multiplicity $q^2-1$,
which implying that $AG(2, q)$ is a Ramanujan hypergraph.
\end{exm}

\begin{proof}[Proof of Theorem \ref{2RR}]
By Theorem \ref{spec} and Theorem \ref{onesR}, if we have a right-sided Ramanujan hypergraph, we can construct an infinite tower of right-sided Ramanujan $2$-coverings.
Note that a projective plane or affine plane of order $q$ always exists if $q$ is a prime power.
So, from the Ramanujan hypergraphs in Example \ref{exm}, we get the desired result immediately.
\end{proof}

In the Definition \ref{RamaGraph} of Ramanujan hypergraphs,
if $d=r$, then the lower bound in (\ref{RamaHR}) is reduced to
$\la \ge -d$, which is a trivial condition.
So, a $(d,d)$-regular hypergraph is Ramanujan if and only if it is right-sided Ramanujan.

\begin{cor}\label{existRdd}
There exists an infinite family of $(d,d)$-regular Ramanujan hypergraphs for every $d\ge3$.
\end{cor}

\begin{rmk}\label{Conc}
Marcus, Spielman and Srivastava have proved every $(d,r)$-biregular graph $G$ has a Ramanujan $2$-covering, namely, there exists a $2$-cover $\bar{G}$ of $G$ such that its new eigenvalues are at most $\sqrt{d-1}+\sqrt{r-1}$ (\cite[proof of Theorem 5.6]{MSS}).
Frome the graphs $G$ and $\bar{G}$, we get a $(d,r)$-regular hypergraph $H$ and its $2$-cover $\bar{H}$ such that their incidence graphs are respectively $G$ and $\bar{G}$.
By hypergraph covering lemma (Lemma \ref{hypcov}) and spectral union property (Theorem \ref{inclusion} and Bilu-Linial \cite{BL}),
the new eigenvalues of $\bar{H}$ are at most $$\left(\sqrt{d-1}+\sqrt{r-1}\right)^2-d=r-2+2\sqrt{(d-1)(r-1)}.$$
So, every $(d,r)$-regular hypergraph has a right-sided Ramanujan $2$-covering, implying Theorem \ref{onesR}.

Marcus, Spielman and Srivastava proved there exists an infinite sequence of $d$-regular bipartite Ramanujan graphs (\cite[Theorem 5.5]{MSS}), which implies Corollary \ref{existRdd} by a similar discussion as above.
The starting $d$-regular bipartite graph for construction in \cite[Theorem 5.5]{MSS} is a complete bipartite graph $K_{d,d}$ of degree $d$.
Note that $K_{d,d}$ is corresponding to a $d$-uniform hypergraph with only one edge of multiplicity $d$.
Here we use complete $d$-uniform hypergraph $K_{d+1}^d$ for the starting hypergraph in the construction.
\end{rmk}

\subsection{Left-sided Ramanujan hypergraphs}

We now discuss the $(d,r)$-regular left-sided Ramanujan hypergraphs $H$ for $d \ne r$.
Note that  $d<r$ if and only if $\nu(H) > e(H)$.
Define $\tau_H=\min\{\nu(H),e(H)\}$.
Then the lower bound in (\ref{RamaHR}) is exactly
\begin{equation}\label{lowLR}\la_{\tau_H}(H) \ge r-2-2\sqrt{(d-1)(r-1)}.
\end{equation}


\begin{proof}[Proof of Theorem \ref{onesL}]
By Corollary \ref{Match3},
there exists an assignment $\phi: E(\overleftrightarrow{B_H}) \to \S_2$ such that the $\tau_H$-th largest eigenvalue $\la_{\tau_H}(A(B_H,\phi))$ is at least $\mu_{\tau_H}(B_H)$.
So, if $\mu_{\tau_H}(B_H)\ge \left|\sqrt{d-1}-\sqrt{r-1}\right|$, we have
\begin{align*}
\la_{\tau_H}(A(H,\phi)) & = \la_{\tau_H}(B_H,\phi)^2 -d \ge \mu_{\tau_H}(B_H)^2-d \\
& \ge \left(\sqrt{d-1}-\sqrt{r-1}\right)^2-d=r-2-2\sqrt{(d-1)(r-1)},
\end{align*}
which implies that $H$ has a left-sided Ramanujan $2$-covering $H^\phi_B$ by Eq. (\ref{lowLR}).
\end{proof}

If $H$ is a left-sided Ramanujan hypergraph satisfying Eq. (\ref{conj}), then
$H_B^\phi$ is a left-sided Ramanujan hypergraph for some $\phi: E(\overleftrightarrow{B_H}) \to \S_2$, as $\tau_{H_B^\phi}=2\tau_H$ and
$$\la_{\tau_{H^\phi}} \ge \min\{\la_{\tau_H}(H),\la_{\tau_H}(A(H,\phi))\}.$$
We will further investigate the ``middle root'' of the matching polynomial of biregular graphs.
Let $G$ be a  $(d,r)$-biregular graph, and let $V_1,V_2$ be the two parts of $G$ consisting of vertices of degree $d$ and $r$ respectively.
Denote $\tau_G:= \min \{|V_1|,|V_2|\}$.

\begin{lem}\label{mutau}
Let $G$ be a  $(d,r)$-biregular graph, and let $\mu_{\tau_G}(G)$ be the $\tau_G$-th largest root of $\mu_G(x)$.
Then $G$ has a $\tau_G$-matching, and
$$0< \mu_{\tau_G}(G) \le \max\{\sqrt{d}, \sqrt{r}\}.$$
\end{lem}

\begin{proof}
Let $V_1,V_2$ be the two parts of $G$ consisting of vertices of degree $d$ and $r$ respectively.
Then $d |V_1|=r |V_2|$.
Assume first that $d \ge r$ (or $|V_1| \le |V_2|$).
We assert that $G$ has a matching covering all vertices of $V_1$, which implies that
$G$ has a $\tau_G$-matching as $|V_1|=\tau_G$.
For any nonempty subset $U$ of $V_1$, letting $N(U)$ be the set of neighbors of the vertices of $U$, surely $N(U) \subseteq V_2$, and
$  d \cdot |U| \le r\cdot |N(U)|.$
So we have $| N(U)| \ge |U|$ as $d \ge r$, and hence arrive at the assertion by Hall’s Marriage Theorem.

Denote $n_1:=|V_1|=\tau_G$, and $n_2:=|V_2|$.
The matching polynomial of $G$ can be written as
\begin{equation}\label{MatChar}
\begin{split}
 \mu_G(x)&=\sum_{i=0}^{n_1} (-1)^i m_i x^{n_1+n_2-2i}=x^{n_2-n_1} \sum_{i=0}^{n_1} (-1)^i m_i (x^2)^{n_1-i}\\
 &=x^{n_2-n_1} (x^2-\mu_1(G)^2)\cdots (x^2-\mu_{n_1}(G)^2),
 \end{split}
 \end{equation}
where $\mu_1(G), \ldots, \mu_{n_1}(G)$ are the first $n_1$ nonnegative largest roots of $\mu_G(x)$ arranged in non-increasing order.
As $G$ has an $n_1$-matching,
$$\mu_1(G)^2 \cdots  \mu_{n_1}(G)^2=m_{n_1}>0,$$ and then $\mu_{n_1}(G)=\mu_{\tau_G}(G)>0$.
Noting that $m_1$ is exactly the number of edges of $G$, we have
$$n_1 \mu_{n_1}(G)^2 \le \mu_1(G)^2 + \cdots + \mu_{n_1}(G)^2=m_{1}=n_1 d,$$
and hence $\mu_{n_1}(G)^2 \le d$.

Similarly, if $r \ge d$, then $n_2 \le n_1$ and $G$ has an $n_2$-matching covering all vertices of $V_2$.
In this case, $\mu_{n_2}(G)=\mu_{\tau_G}(G)>0$, and $\mu_{n_2}(G)^2 \le r$.
The result now follows.
\end{proof}

By Lemma \ref{mutau} we know that $\mu_{\tau_H}(B_H)>0$ for a $(d,r)$-regular hypergraph $H$.
Note that $\mu_{B_H}(x)$ has zero root with multiplicity $|\nu(H)-e(H)|$.
So, the lower bound of $\mu_{\tau_H}(B_H)$ in Eq. (\ref{conj}) is equivalent to say  that
 $\mu_{B_H}(x)$ has no roots $\mu$ with $0 < |\mu|< |\sqrt{d-1}-\sqrt{r-1}|$.
Note that if $d=r$, the condition of Eq. (\ref{conj}) is trivial.
So, we pose the following problem:

{\bf Problem 1.} Does Eq. (\ref{conj}) hold for every $(d,r)$-regular hypergraph or $(d,r)$-regular left-sided Ramanujan hypergraph $H$ with $d \ne r$?

Finally, we give an explicit construction of infinitely many $(q+1,q)$-regular left-sided Ramanujan hypergraphs, where $q$ is a prime power greater than $4$.
We need some knowledge for preparation.
Let $G$ be a simple graph and $G^\phi$ be a $k$-fold covering of $G$,
where $\phi: E(\overleftrightarrow{G})\to\mathbb{S}_k$.
For each permutation $\s \in \mathbb{S}_k$, we associate it with a permutation matrix $P_{\s}=(p_{ij})$ of order $k$ such that $p_{ij}=1$ if and only if $i=\s(j)$.
In fact, the above mapping $\s \to P_{\s}$ is the permutation representation of $\mathbb{S}_k$.
Adopting the terminology from \cite{MoharT}, $G^\phi$ is called an \emph{Abelian lift} (or \emph{Abelian covering}) if all permutations $\s$ or representation matrices $P_{\s}$ commute with each other for all $\s \in \phi(E(\overleftrightarrow{G}))$.
 Mohar and Tayfeh-Rezaie \cite{MoharT} used the representation matrices to define the covering graph $G^\phi$, which are consistent with the covering graph defined in this paper.

Suppose that $G^\phi$ is an Abelian $k$-lift of $G$.
As the matrices $P_{\s}$ commute with each other for all $\s \in \phi(E(\overleftrightarrow{G}))$, they have a common basis of
$k$ eigenvectors  denoted by $\mathcal{B}_\phi$.
Mohar and Tayfeh-Rezaie \cite{MoharT} proved that a beautiful result on the
spectrum of $A(G^\phi)$, where $\la_x(A)$ denotes the eigenvalue of $A$ corresponding to $x$ if $x$ is an eigenvector of $A$.

\begin{lem}\cite{MoharT}\label{Mspec}
Let $G$ be a simple graph and $G^\phi$ be an Abelian $k$-lift of $G$, where $\phi: E(\overleftrightarrow{G})\to\mathbb{S}_k$.
For each $x \in \mathcal{B}_\phi$, let $A_x$ be the matrix obtained from the adjacency matrix $A(G)$ of $G$ by replacing each nonzero $(u,v)$-entry by $\la_x(P_{\phi(u,v)})$.
Then the spectrum of $A(G^\phi)$ is the multi-set union of the spectra of $A_x$ for all $x \in \mathcal{B}_\phi$.
\end{lem}

\begin{proof}[Proof of Theorem \ref{LR-const}]
Let $H$ be the hypergraph corresponding to the affine plane $AG(2,q)$, where $q$ is a prime power $q$ greater than $4$.
Let $B_H$ be the incidence graph of $H$, which is a $(q+1,q)$-biregular graph.
Let $\{v_0,e_0\}$ be a fixed edge of $B_H$, where $v_0 \in V(HG)$, $e_0 \in E(H)$ and $v_0 \in e_0$.
Let $\phi: E(\overleftrightarrow{B_H})\to\mathbb{S}_k$ such that $\phi((v_0,e_0))=(12\cdots k)=:\s_0$ (a cyclic permutation of $\mathbb{S}_k$), $\phi((e_0,v_0))=\s_0^{-1}$, and $\phi(v,e)=1$ for all other arcs of $(v,e) \in  E(\overleftrightarrow{G})$.
In this case, $\mathcal{B}_\phi$ consists of $k$ eigenvectors $x_j$ of $P_{\s_0}$ corresponding to the eigenvalues $\la_{x_j}(P_{\s_0})=e^{{\bf i} 2 \pi j/k}$ for $j=0,1,\ldots,k-1$, where ${\bf i}=\sqrt{-1}$.
Surely, $\la_{x_j}(P_{\s_0^{-1}})=\la_{x_j}(P^{-1}_{\s_0})=e^{-{\bf i}2 \pi j/k}$ for  $j=0,1,\ldots,k-1$.
For each  $j=0,1,\ldots,k-1$, let $A_{x_j}$ be the matrix obtained from the $A(B_H)$ only by replacing the $(v_0, e_0)$-entry by $e^{{\bf i} 2 \pi j/k}$ and the $(e_0, v_0)$-entry by $e^{-{\bf i} 2 \pi j/k}$. Surely, $A_{x_0}=A(B_H)$.

Observe that $A(B_H)$ has eigenvalues $\pm \sqrt{q^2+q}$ respectively with multiplicity $1$, $\pm \sqrt{q}$ respectively with multiplicity $q^2-1$, and $0$ with multiplicity $q$.
Noting that $\tau_H=q^2$, so
$\la_{\tau_H}(A(B_H))=\la_{q^2}(A(B_H))=\sqrt{q}$.
Write $A_{x_j}$ as the form:
$$ A_{x_j}=A(B_H) + B_j,$$
where, after a suitable labeling of the vertices,
$$B_j=\left(\begin{array}{cc} 0 & e^{{\bf i} 2 \pi j/k}-1 \\
e^{-{\bf i} 2 \pi j/k}-1 & 0 \end{array}\right) \oplus O_{2q^2+q-2},$$
where $O_{2q^2+q-2}$ is a zero matrix of order $2q^2+q-2$.
So, when $q \ge 5$, for each  $j=1,\ldots,k-1$,
$$ \la_{q^2}(A_{x_j}) \ge \la_{q^2}(A(B_H))+ \la_{2q^2+q}(B_j)=\sqrt{q}-\sqrt{2(1-\cos(2 \pi j/k)}\ge \sqrt{q}-\sqrt{q-1}.$$
By Lemma \ref{Mspec},
the spectrum of $A(B_H^\phi)$ is a multiset union of the spectra of $A(B_H)$ and $A(B_{x_j})$ for $j=1,2,\ldots,k-1$.
So,
$$ \la_{kq^2}(A(B_H^\phi)) \ge \min\{ \la_{q^2}(A(B_H)), \la_{q^2}(A_{x_j}): j=1,2,\ldots,k-1\}\ge \sqrt{q}-\sqrt{q-1}.$$
Note that $\tau_{H_B^\phi}=k \tau_H=kq^2$.
So,
\begin{align*}
 \la_{\tau_{H_B^\phi}}(A(H_B^\phi))& =\la_{kq^2}(A(H_B^\phi)) = \la_{kq^2}(A(B_H^\phi))^2-(q+1) \\
 & \ge  \left(\sqrt{q}-\sqrt{q-1}\right)^2-(q+1)=q-2-2\sqrt{q(q-1)}.
 \end{align*}
By definition, $H_B^\phi$ is left-sided Ramanujan for all positive integers $k$ with $\phi: E(\overleftrightarrow{B_G})\to\mathbb{S}_k$ defined as in the above.
\end{proof}


Maybe like the difficulty of characterizing non-bipartite Ramanujan $d$-regular graphs,
the characterization of a $(d,r)$-regular hypergraph being both left-sided and right-sided Ramanujan hypergraph seems very difficult for $d \ne r$.

\end{document}